\documentclass[letterpaper]{article}
\usepackage[T1]{fontenc}
\usepackage[english]{babel}
\usepackage[top=3cm, bottom=3cm, left=3.85cm, right=3.85cm]{geometry}
\usepackage[onehalfspacing]{setspace}
\usepackage{amsmath,amssymb,amsthm,wasysym}
\usepackage{graphicx,hyperref,tikz,tikz-cd}
\usepackage{mathptmx}

\newtheoremstyle{xxx}
  {2ex}{0}
  {}{0pt}{\bf}{.}{ }{}

\theoremstyle{xxx}
\newtheorem{thrm}{Theorem}
\newtheorem{corl}[thrm]{Corollary}
\newtheorem{lmma}[thrm]{Lemma}

\newcommand{\diam}{\textnormal{diam}}
\newcommand{\seq}[1]{\{#1_n\}}
\newcommand{\pfin}[1]{\mathcal{P}_\textnormal{fin}(#1)}

\newcommand{\defn}[1]{\emph{#1}}

\title{Diversities and Conformities\footnote{Research funded in part by NSERC.}}
\author{Andrew Poelstra\footnote{Simon Fraser University, Burnaby, British Columbia, V5A 1S6, Canada. \texttt{asp11@sfu.ca}}}
\date{}
\begin{document}

\maketitle

\begin{abstract} Diversities have recently been developed as multiway
metrics admitting clear and useful notions of hyperconvexity and tight span. In
this note we consider the analytic properties of diversities, in particular the
generalizations of uniform continuity, uniform convergence, Cauchy sequences and
completeness to diversities. We develop conformities, a diversity analogue of
uniform spaces, which abstract these concepts in the metric case.
We show that much of the theory of uniform spaces admits a natural analogue in
this new structure; for example, conformities can be defined either axiomatically
or in terms of uniformly continuous pseudodiversities.

Just as diversities can be restricted to metrics, conformities can be restricted
to uniformities. We find that these two notions of restriction, which are functors
in the appropriate categories, are related by a natural transformation.
\end{abstract}

\section{Introduction}

The theory of metric spaces is well-understood and forms the basis of much of
modern analysis. In 1956, Aronszajn and Panitchpakdi developed the notion of
\emph{hyperconvex} metric spaces \cite{aronszajn+panitchpakdi1956} in order to
apply the Hahn-Banach theorem in a more general setting. In fact, every metric
space can be embedded isometrically in a minimal hyperconvex space, as discovered
by J. R. Isbell \cite{isbell1964} (as the ``hyperconvex hull'') and later by
A. W. M. Dress \cite{dress1984} (as the ``metric tight span'').

These minimal hyperconvex spaces, or \emph{tight spans}, proved to be powerful
tools for the analysis of finite metric spaces. The theory of tight spans, or
\emph{T-theory}, is overviewed in \cite{dress+moulton+terhalle1996}. Its history,
as well as applications to phylogeny, are given in \cite{bryant+tupper2012}.

In light of these applications of T-theory, D. Bryant and P. Tupper developed the
theory of \emph{diversities} alongside an associated tight span theory in
\cite{bryant+tupper2012}. Diversities are multiway metrics mapping
finite subsets of a ground space $X$ to the nonnegative reals. The axioms were
chosen based on their specific applications to phylogeny (where they
had already appeared in special cases) and their ability to admit a tight span theory.
This diversity tight span theory contains the metric tight span theory as a special
case (using so-called diameter diversities), but also allows new behavior which may
be useful in situations such as microbial phylogeny, where the idea of a historical
``phylogenetic tree'' does not make sense. Several examples, along with pictures,
of this phenomenon are given in \cite{bryant+tupper2012}.

A classic paper by Andr\'e Weil \cite{weil1937} developed the theory of \emph{uniform
spaces}, which generalize metric spaces. Uniform spaces admit notions of uniform
continuity, uniform convergence and completeness which coincide with the standard
notions when metric spaces are considered as uniform spaces. This theory has been
described in Bourbaki's \emph{General Toplogy} \cite{bourbaki1989} as well as John
Kelley's classic text \cite{kelley1955}. The metric topology can be derived purely
from properties of the uniform space (via the so-called \emph{uniform topology}),
and in this sense uniform spaces lie ``between'' metric spaces and topologies.

In this note, we develop \emph{conformities}, which generalize diversities in
analogy to Weil's uniform space generalization of metrics. We will describe
uniform continuity, uniform convergence, Cauchy sequences and completeness
for diversities, and show that these can be characterized in terms of
conformities, giving an abstract framework in which to analyze the uniform
structure of diversities. This is motivated by the observation that while
diversities generalize metric spaces in a straightforward way (in fact they
restrict to metric spaces), they can exhibit very non-smooth behavior with
respect to these spaces (c.f. our Theorem \ref{steiner-nonsense}). Therefore
the existing tools for metric spaces are insufficient to get a handle on the
behavior of diversities.


\section{Preliminaries}

Throughout this paper, we will denote the \emph{finite power set} of a 
given set $X$ by
\[ \pfin{X} = \{ A\subseteq X: |A|<\infty\}. \]

We begin with the Bryant-Tupper definition from \cite{bryant+tupper2012}: a
\defn{diversity} is a pair $(X,\delta)$ where $X$ is some set and $\delta:\pfin{X}\to\mathbb{R}$
is a function satisfying
\begin{enumerate}
\item[D1.] If $A\in\pfin{X}$, $\delta(A) \geq 0$ and $\delta(A) = 0$ iff $|A| \leq 1$.
\item[D2.] If $A,B,C\in\pfin{X}$ with $C\neq\varnothing$, then
\[ \delta(A \cup B) \leq \delta(A \cup C) + \delta(C \cup B) \]
\end{enumerate}
If for some $A\in\pfin{X}$, $\delta(A) = 0$ but $|A| > 1$ we have the weaker
notion of a \defn{pseudodiversity}. It is shown in \cite{bryant+tupper2012} from
these axioms that if $A\subseteq B$ then $\delta(A) \leq \delta(B)$, i.e.,
(pseudo)diversities are monotonic, and that the restriction of a diversity to
sets of size 2 forms a pseudometric $d(x,y) = \delta(\{x,y\})$. We call this metric
the \defn{induced metric} of the diversity.

For a metric space $(X,d)$, there are two important diversities on $X$
having $d$ as an induced metric:
\begin{itemize}
\item The \defn{diameter diversity} $(X,\diam_d)$ defined by
\[ \diam_d(\{x_1,\ldots,x_n\}) = \sup_{i,j} d(x_i, x_j) \]
When $X = \mathbb{R}^n$ and $d$ is the Euclidean metric we refer to
this diversity simply by $\diam$.
\item The \defn{Steiner tree diversity} $(X,\delta)$ defined for each
finite set $A\subseteq X$ as the infimum of the size of the minimum
Steiner tree on $A$.

(Recall that a Steiner tree on $A$ is a tree whose vertex set $V$ satisfies
$A\subseteq V\subseteq X$, with each edge $(x,y)$ weighted by $d(x,y)$. The
size of the tree is the sum of its edge weights.)
\end{itemize}

In fact, these examples are the extremes of diversity behavior relative to
their induced metrics, in the sense that for any diversity $(X,\delta')$ which
induces a metric $d$, we have
\[ \diam_d \leq \delta' \leq \delta \]
where $\delta$ is the Steiner tree diversity on $(X,d)$. This can be shown by
a straightforward argument\footnote{Bryant and Tupper, upcoming.}.

To demonstrate the difference between the diameter and Steiner tree diversities,
consider the Euclidean metric $(\mathbb{R}^3,d)$.
The induced metric of both the diameter and Steiner tree diversity is the
Euclidean metric. For any finite set $A$ contained in an
$\epsilon$-ball, $\diam(A) < \epsilon$. To contrast, in any $\epsilon$-ball
we can find finite sets $A$ for which $\delta(A)$ is arbitrarily large:
\begin{thrm}\label{steiner-nonsense} The Steiner tree diversity function
$\delta$ on $\mathbb{R}^3$ is unbounded on every open set of the Euclidean
topology.
\end{thrm}
\begin{proof} Without loss of generality, we show the result for
$\epsilon$-balls about 0. For each $n\in\mathbb{N}$, define
\[ G_n = \left\{ \left(\frac{i}{n^2}, \frac{j}{n^2}, \frac{k}{n^2}\right) : 0 \leq i,j,k < n \right\} \]
which is a grid of points contained in the cube $[0,1/n]^3$. Since there are
$n^3$ points, a minimum spanning tree connecting the members of $G_n$ must
have $n^3 - 1$ edges, each of length $\geq 1/n^2$, since that is the least
distance between two points. Therefore the size of the minimum spanning tree
on $G_n$ is at least $(n^3 - 1)/n^2$, which can be taken as large as we like
by taking $n$ large enough. Since the minimal Steiner tree on $G_n$ has size
at least 0.615 times that of the minimal spanning tree \cite{du1991}, we
have $\delta(G_n)\to\infty$ as $n\to\infty$ even though $\diam(G_n)\to0$.
\end{proof}

A similar construction for the Steiner tree diversity on $\mathbb{R}^2$
gives sets of diversity $(0.615 - \epsilon)$ for every $\epsilon > 0$ in
every Euclidean ball. On $\mathbb{R}$ the Steiner tree diversity and
diameter diversity are identical.
The dramatic difference between the many-point behavior of these two diversities
in dimension 2 or higher demonstrates that diversities are not characterized by
their induced metrics, even up to a constant.


In this section and the next, we will define uniform convergence, uniform
continuity and completeness explicitly in terms of an underlying diversity;
in Section~\ref{conf} we will describe \emph{conformities}, which abstract these
properties for diversities.This is in analogy to Weil's uniformities, which
abstract the same concepts for metric spaces.

With this goal in mind, we start with the following definitions: let $(X,\delta_X)$
and $(Y,\delta_Y)$ be diversities. Given $x\in X$, a sequence $\{x_n\} \subset X$
\defn{converges to $x$}, denoted $x_n\to x$, if
\[ \lim_{N\to\infty}\sup_{i_1,i_2,\ldots,i_n\geq N} \delta_X(\{x, x_{i_1}, x_{i_2},\ldots,x_{i_n}\}) = 0 \]
The sequence $\{x_n\}$ is a \defn{Cauchy sequence} if
\[ \lim_{N\to\infty}\sup_{i_1,i_2,\ldots,i_n\geq N} \delta_X(\{x_{i_1}, x_{i_2},\ldots,x_{i_n}\}) = 0 \]
It is an easy consequence of these definitions and the diversity axioms
that limits are unique and every convergent sequence is Cauchy. If every
Cauchy sequence is convergent, we call the diversity \defn{complete}.

Finally, if $f:X\to Y$ is a function such that for every $\epsilon > 0$,
there exists some $d > 0$ such that $\delta_X(A) < d \implies \delta_Y(f(A))
< \epsilon$ for every $A\in\pfin{X}$, we say $f$ is \defn{uniformly
continuous}.

It is not hard to see that for diameter diversities, these definitions
coincide exactly with the standard ones on the induced metric.

For the second half of the paper, we will work extensively with filters,
so we state the definition here:
given a ground set $X$, define a \defn{filter} as a collection $\mathcal{F}$
of subsets of $X$ satisfying $A\cap B$ whenever $A$, $B$ are in $\mathcal{F}$,
and $B\in\mathcal{F}$ whenever $B\supseteq A$ and $A\in\mathcal{F}$. A
\defn{filter base} becomes a filter when all supersets of its elements are
added, in which case we say the base \defn{generates} the filter.

In this paper, we additionally require that $\varnothing\notin\mathcal{F}$.

\section{Comparison with metrics}

In this section, we contrast the uniform convergence of sequences with
respect to diversities and their induced metrics. In particular, we show that although the
Cauchy property for sequences is much stronger for diversities (we
demonstrate a sequence which is not Cauchy with respect to a diversity,
even though it is Cauchy with respect to the induced metric), completeness
of a diversity is equivalent to completeness of its induced metric.
This tells us that every diversity which induces a Euclidean metric (e.g.,
the Steiner tree diversity on $\mathbb{R}^n$) is complete.

Since the set of Cauchy sequences in a diversity may be smaller than
the set of Cauchy sequences of its induced metric, this may provide a
simpler way to determine completeness of metric spaces.
%
%

At the end of the section, we construct the analogue of completion for
diversities.

\subsection{Completeness in diversities and metric spaces}

\begin{thrm}Let $(X,\delta)$ be a diversity, $d$ its induced metric.
If $(X,d)$ is a complete metric space, then $(X,\delta)$ is a complete
diversity.\end{thrm}
\begin{proof} Suppose that $(X,d)$ is complete. Let $\{ x_n \}$ be a Cauchy
sequence in $(X,\delta)$. Then it is also Cauchy in $(X,d)$, and therefore
converges to some element $x$. We claim that $x_n\to x$ in $(X,\delta)$.
To this end, let $\epsilon > 0$. Then there exists $N$ such that:
\begin{itemize}
\item $d(x_n,x) < \epsilon$ for all $n > N$ (since $x_n\to x$ in $(X,d)$)
\item $\delta (\{x_{n_1},x_{n_2},\ldots,x_{n_m}\}) < \epsilon$ for all $n_i > N$
(since $\{ x_n \}$ is Cauchy in $(X,\delta)$).
\end{itemize}
Therefore, for all $n_1,\ldots,n_m > N$,
\begin{align*}
\delta(\{x, x_{n_1},\ldots, x_{n_m}\}) &\leq \delta(\{ x, x_{n_1}\}) + \delta(\{x_{n_1},\ldots, x_{n_m}\})\\
  &= d(x, x_{n_1}) + \delta(\{x_{n_1},\ldots, x_{n_m}\}) < 2\epsilon
\end{align*}
i.e., $x_n\to x$ in $(X,\delta)$.
\end{proof}

As mentioned, the set of Cauchy sequences in a diversity may be strictly
smaller than the set of Cauchy sequences in the induced metric. For example,
let $(X,\delta)$ be the Steiner tree diversity on $\mathbb{R}^3$, and
consider the sets $\{G_n\}_{n\in\mathbb{N}}$ from Theorem \ref{steiner-nonsense}.

Order each set $G_n$ somehow and define the sequence $\{x_i\}$ by
concatenating them, i.e.,
\[ \{ x_n \} = G_1G_2G_3\cdots \]
which is Cauchy in the induced metric of $(X,\delta)$ (since eventually every pair
of points is confined to arbitrarily small cubes $[0,\epsilon]^3$). However, it
is not Cauchy in $(X,\delta)$, since we saw in the proof of Theorem
\ref{steiner-nonsense} that $\delta(G_n)$ becomes arbitrarily large
as $n\to\infty$. In other words, every tail of $\{x_n\}$ has
arbitrarily large finite sets, so $\{x_n\}$ is not Cauchy.

In light of this example, it is interesting to know that every complete diversity
has a complete induced metric, which is proved with the following Lemma:

\begin{lmma} Let $(X, \delta)$ be a diversity, $d$ its induced metric.
Let $\{ x_n \}$ be Cauchy in $(X, d)$. Then it has a subsequence that
is Cauchy in $(X,\delta)$.\label{claus}\end{lmma}
\begin{proof} Define the subsequence $\{ x_{n_i} \}$ by
\[ n_i = \min \{ n: d(x_n, x_m) < 2^{-i} \text{ for all $m \geq n$} \} \]
Given $\epsilon > 0$, choose $N$ such that $2^{1-N} < \epsilon$.
Then for all $i_1 \leq i_2 \leq \cdots \leq i_m$ greater than $N$,
\begin{align*}
\delta(\{ x_{i_1}, \ldots, x_{i_m} \})
  &\leq \delta(\{x_{i_1}, x_{i_2}\}) + \cdots + \delta(\{x_{i_{m-1}}, x_{i_m}\})	\\
  &< 1/2^{i_1} + \cdots + 1/2^{i_m}	\\
  &< \sum_{i = N}^\infty 1/2^i	\\
  &= 2^{1-N} < \epsilon
\end{align*}
That is, $\{ x_{n_i} \}$ is Cauchy in $(X,\delta)$.
\end{proof}

\begin{thrm}Let $(X,\delta)$ be a diversity, $d$ its induced metric.
If $(X,\delta)$ is a complete diversity, then $(X,d)$ is a complete
metric space.\end{thrm}
\begin{proof} Let $\{ x_n \}$ be a Cauchy sequence in $(X, d)$. Then
by Lemma \ref{claus} it has a subsequence $\{ x_{i_n} \}$ that is Cauchy
in $(X,\delta)$, which converges to some element $x$ since the diversity
is complete. (It converges in both $(X,\delta)$ and $(X,d)$.)

Then $x_n$ converges to $x$ in $(X,d)$, since for any $\epsilon$ we have
$d(x_n, x) \leq d(x_n, x_{i_m}) + d(x_{i_m}, x) < 2\epsilon$ for $m,n$
large enough.
\end{proof}

\subsection{Completion\label{sect-comp}}

In light of the equivalence between metric completeness and diversity
completeness, it is perhaps not so surprising that every diversity can
be completed in a canonical way. To do so, we require two more definitions
from \cite{bryant+tupper2012}: an \defn{embedding} $\pi:Y_1\to Y_2$ is
an injective map between diversities $(Y_1,\delta_1)$ and $(Y_2,\delta_2)$
such that $\delta_1(A) = \delta_2(\pi(A))$ for all $A\in\pfin{Y_1}$. A
\defn{isomorphism} is a surjective embedding.

\begin{thrm}Every diversity $(X,\delta)$ can be embedded in a complete diversity.\end{thrm}
\begin{proof} Let $\hat{X}$ be the set of all Cauchy sequences in $X$. Identify
any two sequences $\{x_i\}, \{ y_i \}$ which satisfy $\lim_{n\to\infty}\delta(\{x_n,y_n\})=0$
(so $\hat{X}$ is actually a set of equivalence classes). Define the function
$\hat{\delta}$ from $\mathcal{P}_\text{fin}(\hat{X})\to\mathbb{R}$ by

\[ \hat{\delta}(\{ \{x^1_i\}, \{x^2_i\}, \ldots, \{x^n_i\} \}) = \lim_{N\to\infty} \sup_{i_1,\ldots,i_n\geq N} \delta(\{ x^1_{i_i}, x^2_{i_2},\ldots, x^n_{i_n} \}) \]

It can then be shown that $(\hat{X}, \hat{\delta})$ is a complete diversity,
and that the map $x \mapsto \{ x, x, x, \ldots \}$ from $(X,\delta)$ is an
embedding. The proof is an exercise in notation.
\end{proof}

This completion is dense in the sense that every member $x$ of $\hat{X}$ has a
sequence $\{x_i\}\subseteq X$ with $x_i\to x$ in $\hat{X}$. (Let $\{y_i\}$ be
a representative of $x$ and define $x_i = \{y_1,y_2,\ldots,y_i,y_i,y_i,\ldots\}$.)
It also satisfies a universal property analogous to that for metric completion:
\begin{thrm} Let $(X,\delta)$ be a diversity, $(\hat{X}, \hat{\delta})$ its
completion. Then for any complete diversity $(Y, \gamma)$ and any uniformly
continuous function $f:X\to Y$, there is a unique uniformly continuous
function $\hat{f}:\hat{X}\to Y$ which extends $f$.\end{thrm}
\begin{proof} Let $\{x_i\}$ be a representative sequence of some member of
$\hat{X}$, and define $\hat{f}(\{x_i\}) = \lim_{i\to\infty} f(x_i)$, which
is defined and independent of representative since $f$ is uniformly continuous
and $Y$ is complete. To show $\hat{f}$ is uniformly continuous, pick $\epsilon
> 0$ and $d > 0$ such that $\gamma(f(A)) < \epsilon$ whenever $\delta(A) < d$
for all $A\in\pfin{X}$. Then for all $B=\{\{x_i^1\},\{x_i^2\},\ldots,\{x_i^m\}\}
\in\pfin{\hat{X}}$ with $\hat{\delta}(B) < d/2$, we have $\gamma(\hat{f}(B))
= \gamma( \{\lim_{i\to\infty} f(x^n_i)\}_{n=1}^{m})) < \epsilon$ since for
large enough $N$, $\delta(\{f(x^n_N)\}_{n=1}^{m}) < 3d/4$.

To show uniqueness of $\hat{f}$, let $\hat{g}$ be another uniformly continuous
function extending $f$ to $\hat{X}$. For all $x\in\hat{X}$ we have $\{x_i\}
\subset X$ with $x_i\to x$ in $\hat{X}$, and by uniform continuity $\hat{g}(x) =
\lim_{i\to\infty} f(x_i) = \hat{f}(x)$.
\end{proof}

This is a universal property in the sense that for every complete diversity
$\hat{X}'$ extending $X$ and having the property, there is an isomorphism
$\tilde{j}:\hat{X}'\to\hat{X}$. (Specifically, let $\tilde{j}$ be the unique
uniformly continuous extension of the identity map $j:X\to\hat{X}$ to $\hat{X}'$.)

\section{Conformities \label{conf}}

In this section we introduce a generalization of diversities analogous to
\emph{uniformities}, which generalize metric spaces. Uniformities lie between
metric spaces and topologies, in the sense that every metric space defines a
uniformity, and every uniformity defines a topology (which coincides with the
metric topology when the uniformity came from a metric).
Uniformities characterize uniform continuity, uniform convergence and Cauchy
sequences, which are not topological concepts.

The carry-over from the metric case is natural 
but nontrivial, since diversities can behave differently on sets of different
cardinality.
Since this construction is qualitatively different from metric uniformities,
it requires a different name. We asked ourselves ``what would you call a uniformity
that came from a diversity?'', and the answer was clear: a \emph{conformity}.

Throughout this section, we will give the analogous definitions and results
for uniformities, using the standard treatment from Kelley \cite{kelley1955}.
We begin by defining conformities and comparing them to uniformities; we show
that just like uniformities, conformities have a countable base if and only
if they are generated by some pseudodiversity.

We then briefly touch on the problem of completion for conformities.

Finally, we define power conformities: from a conformity defined on a set
$X$, we can construct a conformity on $\pfin{X}$ from which pseudodiversities
can be considered uniformly continuous functions. We show that every conformity
is generated by exactly the set of pseudodiversities which are uniformly
continuous from its power conformity to $\mathbb{R}$. This gives an equivalent
definition of conformity in terms of pseudodiversities.

\subsection{Conformities of diversities}

Recall that for $(X,d)$ a metric space, $\{x_n\}$ a sequence in $X$, that $\{x_n\}$
is Cauchy iff for each $\epsilon > 0$ there is some $N$ such that every pair of
points $(x_i,x_j)$ with $i>N,j>N$ has $d(x_i,x_j) < \epsilon$.

Similarly, let $f:X\to Y$ be a function between metric spaces $(X,d)$ and
$(Y,g)$. Then $f$ is uniformly continuous iff each $\epsilon > 0$ has a
$\delta > 0$ such that whenever pairs of points $(x,y)\in X\times X$ satisfy
$d(x,y) < \delta$, the pairs $(f(x), f(y))$ satisfy $g(f(x),f(y)) < \epsilon$.

A similar characterization of uniform convergence of sequences of functions can
be given in terms of pairs of points. From these observations arises the theory
of uniformities, which is described in any standard text on analysis (c.f.
\cite{bourbaki1989,kelley1955}). We briefly describe the theory here. For any set
$X$ define a \defn{uniformity} on $X$ as a filter $\mathcal{U}$ on $X\times X$
satisfying
\begin{enumerate}
\item[U1.] $(x, x) \in U$ for every $x \in X$, $U\in\mathcal{U}$.
\item[U2.] If $U\in\mathcal{U}$, $(x,y)\in U$, then $(y,x)\in U$.
\item[U3.] For every $U\in\mathcal{U}$, there exists some $V\in\mathcal{U}$ with
$V\circ V\subseteq U$, where in general we define
\[ U\circ V := \{ (x, z) : (x,y) \in U, (y,z) \in V \text{ for some } y \in X \} \]
\end{enumerate}

In particular, for any pseudometric space $(X,d)$ we can define the \defn{metric
uniformity} as the filter on $X\times X$ defined by
\[ U^\epsilon = \{ (x, y) : d(x,y) < \epsilon \} \]
for each $\epsilon > 0$. We see from this example that (U1) expresses the
requirement that $d(x,x) = 0$ for all $x\in X$, (U2) expresses symmetry,
and (U3) expresses the triangle inequality.

Uniform structure can be defined entirely with respect to uniformities. For
example, given sets $X,Y$ and uniformities $\mathcal{U},\mathcal{V}$ on $X$
and $Y$ respectively, we can call a function $f:X\to Y$ \defn{uniformly continuous}
if $f^{-1}(V) \in \mathcal{U}$ for every $V\in\mathcal{V}$. (Here $f$ acts
on members of $V$ componentwise.) A sequence $\{x_n\}\subset X$ is \defn{Cauchy}
if for every $U\in\mathcal{U}$, there is some $N$ such that pairs of elements
$(x_i,x_j)$ of $\{x_n\}$ are in $U$ whenever $i,j> N$. It is not hard to see
that for metric uniformities, these definitions coincide with the ordinary ones
for metric spaces.

To abstract the uniform structure of diversities, uniformities are clearly
insufficient. For one thing, since diversities map finite sets rather than
pairs, we should seek a filter on $\pfin{X}$ rather than $X\times X$. Then
symmetry is no longer required, but now monotonicity is. Finally, it is not
meaningful to compose finite sets as in (U3), so we will need a different
way to express an analogue of the triangle inequality.

Putting all this together, we define a \defn{conformity} $\mathcal{C}$ on $X$
as a filter on $\pfin{X}$ satisfying
\begin{enumerate}
\item[C1.] $\{x\}\in C$ for every $x \in X$, $C\in\mathcal{C}$.
\item[C2.] For every $C\in\mathcal{C}$, whenever $A\in C$ and $B\subseteq A$,
we have $B\in C$.
\item[C3.] For every $C\in\mathcal{C}$, there exists some $D\in\mathcal{C}$ with
$D\circ D\subseteq C$, where in general we define
\[ U\circ V := \{ u \cup v : u \in U, v \in V \text{ and } u \cap v \neq \varnothing \} \]
\end{enumerate}
Often the term \emph{conformity} is also used to refer to the pair $(X, \mathcal{C})$.

An observation that will be necessary later (one which also holds for uniformities)
is that for any $D\in\mathcal{C}$, $(D\circ D)\circ D = D\circ(D\circ D)$, so that
$D\circ D\circ D$ is defined unambiguously. To estimate the size of this, we also
note that $D\circ D \circ D \subseteq (D\circ D)\circ(D\circ D)$.

As in the metric case, there is a canonical way to generate a conformity from a
diversity; if $\delta$ is a pseudodiversity on $X$, we have the conformity
generated by the sets
\[ C^\epsilon = \{ A: \delta(A) \leq \epsilon \} = \delta^{-1}[0,\epsilon] \]
for each $\epsilon > 0$. (This is equivalent to one using strict inequalities,
but typographically nicer.)

As in the metric case, uniform structure can be defined on conformities in a
way that generalizes that of diversities: let $(X, \mathcal{C})$ and
$(Y, \mathcal{D})$ be conformities. Then a function $f$ is \defn{uniformly
continuous} from $X$ to $Y$ if for all $D\in\mathcal{D}$, the set
$\{ f^{-1}(d):d\in D \}$ is in $\mathcal{C}$. A sequence $\seq{x}$ on $X$ is
a \defn{Cauchy sequence} if for all $C\in\mathcal{C}$, $\pfin{\{x_n\}_{n\geq N}}
\subseteq C$ for some integer $N$. For conformities generated from diversities
in the above way, these definitions coincide with those given in the previous
section.

More generally, given a collection of pseudodiversities $\{\delta_\alpha\}_{\alpha\in\mathcal{A}}$,
we can generate a conformity from the sets
$\left\{\delta^{-1}_\alpha[0, \epsilon]\right\}_{\alpha\in\mathcal{A},\epsilon>0}$. We therefore seek a
characterization of conformities in terms of the diversities which generate
them. (In a later section, we will see that all conformities can be described
in this way, so that we can \emph{define} conformities in terms of such sets.)
We begin by stating a result from Kelley \cite{kelley1955} along with a summary
of his proof:
\begin{thrm} A uniformity is generated by a single pseudometric iff it has a
countable base.\end{thrm}
The standard proof of this theorem goes as follows: it is obvious that any
uniformity generated by a pseudometric has a countable base. Conversely, if
there exists a countable base for a uniformity on $X$, there exists a countable
base $\{U_n\}_{n\in\mathbb{N}}$ for which the following argument holds. Define
the function $f(x,y) = 2^{-n}$, where $n=\sup\{i:(x,y)\in U_i\}$. This generates
the uniformity but does not satisfy the triangle inequality, so define
\[ d(x,y) = \inf \sum_{i=1}^{m-1} f(x_i, x_{i+1}) \]
where the infimum is taken over all sequences $\{x_i\}_{i=1}^m$ with $x_1=x$
and $x_m = y$. This clearly satisfies the triangle inequality, so it just
remains to be shown that $d$ generates the uniformity. This is done by
proving that $d(x,y) \leq f(x,y) \leq 2d(x,y)$, which follows from technical
constraints on $\{U_n\}$.

Given a conformity with a countable base $\{C_n\}$ on a set $X$, one might try
to translate this proof directly: define a function $f(A):\pfin{X}\to\mathbb{R}$
by $f(A) = \sup\{i:A\in C_i\}$, then somehow tweak $f$ to (a) satisfy the
triangle inequality and (b) generate the same conformity as $f$. However, it
appears that any direct analogue to the ``infimum over all paths'' strategy
used in the metric case (there are several) cannot satisfy both (a) and (b)
simultaneously.

Nonetheless, the result is true, which is the content of the next theorem.

\begin{lmma} Let $(X, \mathcal{C})$ have a countable base. Then it has a countable
base $\{ C_n \}$ satisfying $C_0 = \pfin{X}$, $C_i\circ C_i\circ C_i \subseteq C_{i-1}$
for $i>0$.\end{lmma}
\begin{proof} Let $\{ V_n \}$ be a countable base for $\mathcal{C}$. Define
$W_0 = \pfin{X}$, $W_n = V_n \cap W_{n-1}$. Then $\{ W_n \}$ is a nested
countable base. Finally, choose $\{ C_n \}$ as $C_i = W_{n_i}$, where $n_i$
are chosen inductively as $n_0 = 0$, then $(W_{n_i} \circ W_{n_i}) \circ (W_{n_i}
\circ W_{n_i})\subseteq W_{n_{i-1}}$.\end{proof}

\begin{thrm} Let $(X, \mathcal{C})$ be a conformity. There exists a
pseudodiversity $\delta$ which generates $\mathcal{C}$ iff $\mathcal{C}$
has a countable base.\label{idliketobe}\end{thrm}
\begin{proof} If $\delta$ exists the sets $\{ C^{1/n} \}_{n\in\mathbb{N}}$
are our base.

Conversely, let $\{ C_n \}_1^\infty$ be a base for $\mathcal{C}$ satisfying
$C_0 = \pfin{X}$ and $C_i\circ C_i\circ C_i \subseteq C_{i-1}$ for $i > 0$.
Define $\delta'$ on $\pfin{X}$ by
\[ \delta'(A) = \left\{\begin{array}{lr}
  0      & A \in C_n \text{ for all } n\\
  2^{-k} & A \in C_n \text{ for } 0 \leq n \leq k \text{, but } A \notin C_{k+1}
\end{array}\right. \]
Notice that for $k \geq 0$,
\begin{equation} \delta'^{-1}([0, 2^{-k}]) = C_k \label{grasswasgreener} \end{equation}
and that $\delta'$ is monotonic: by (C2), if $A\subseteq B$, then $A\in C_n$
whenever $B$ is.

Define a \defn{chain} as a sequence $\{ A_i \}_{i=1}^n$ in $\pfin{X}$ with
$A_i \cap A_{i-1}\neq\varnothing$ for $i = 2,\ldots,n$. Define a \defn{cycle}
as a chain with $A_1 \cap A_n\neq\varnothing$. Write
\[ \bar{\delta}(A) = \inf_{\substack{\text{chains covering $A$}}} ~ \sum_{i=1}^n \delta'(A_i) \]
\[ \delta(A) = \inf_{\substack{\text{cycles covering $A$}}} ~ \sum_{i=1}^n \delta'(A_i) \]
Notice $\delta(\varnothing) = \bar{\delta}(\varnothing) = 0$.

We claim that $\delta$ is our desired pseudodiversity, since the sets
$(\delta')^{-1}[0, \epsilon]$ generate the conformity, and $\delta\leq\delta'\leq4\delta$.
We prove this in three stages.

\begin{enumerate}
\item [S1.] First of all, $\delta$ is a pseudodiversity. By (C1), for every
$x\in X$, $n\in\mathcal{N}$, $\{x\}\in C_n$ so that $\delta'(\{x\}) = 0$.
Also $\{x\}$ is a cycle covering itself, so $\delta(\{x\}) = 0$.

The triangle equality also holds: let $\epsilon > 0$, $A,C\in\pfin{X}$
and $B\in\pfin{X}$ be nonempty. Choose cycles $\{ A_i \}_1^n$ and $\{ B_i \}_1^m$
covering $A \cup B$ and $B \cup C$, respectively, and for which
\[ \sum_{i=1}^n \delta'(A_i) \leq \delta(A \cup B) + \epsilon \qquad\text{ and } \qquad
   \sum_{i=1}^m \delta'(B_i) \leq \delta(B \cup C) + \epsilon \]
Then $\{ A_i \}_1^n \cup \{ B_i \}_1^m$ forms a cycle (after reordering)
covering $A \cup C$, so
\[ \delta(A \cup C) \leq \sum_{i=1}^n \delta'(A_i) + \sum_{i=1}^m \delta'(B_i) \leq \delta(A \cup B) + \delta(B\cup C) + 2\epsilon \]

\item [S2.] Next, we notice that
\begin{itemize}
\item Every cycle is a chain, so $\delta \leq \bar{\delta}$.
\item If $\{ A_1, \ldots, A_{n-1}, A_n \}$ is a chain, then $\{ A_1,
\ldots, A_{n-1}, A_n, A_{n-1}, \ldots, A_1 \}$ is a cycle --- and the sum
of $\delta'$ over this cycle is less than twice the sum of $\delta'$ over
the original chain. We conclude that
\end{itemize}
\begin{equation}\label{welcome} \delta \leq \bar{\delta} \leq 2\delta \end{equation}

\item [S3.] Finally, we claim that $\bar{\delta}\leq\delta'\leq2\bar{\delta}$.
This combined with \eqref{welcome} will give the main result.

Trivially, $\bar{\delta} \leq \delta'$. For the other inequality, choose
$A \in \pfin{X}$. Our strategy is to induct on the greatest integer $N$
such that $\bar{\delta}(A) < 2^{-N}$.

The case $N = 0$ is easy, because then $\delta' \leq 1 \leq 2\bar{\delta}$.
(This also covers the case $\bar{\delta}(A) = 1$, which is not covered by
the induction.) When $N > 0$, we can choose positive $\epsilon$ less than
$\left(2^{-N} - \bar{\delta}(A)\right)$, and a chain $\{ A_i \}_1^n$ with
\begin{equation} \label{lucy}
 \sum_{i=1}^n \delta'(A_i) < \bar{\delta}(A) + \epsilon < 2^{-N} 
\end{equation}

If $n = 1$, we have $\delta'(A) \leq \delta'(A_1) < 2^{-N} < 2\bar{\delta}(A)$.
Otherwise, there is $k < n$ such that
\begin{equation}
  \label{howiwishyouwerehere}
   \sum_{i=1}^{k-1} \delta'(A_i) \leq \frac{\bar{\delta}(A)}{2} \text{ and }
   \sum_{i=k+1}^n \delta'(A_i) \leq \frac{\bar{\delta}(A)}{2}
\end{equation}
Since $\{A_i\}_1^{k-1}$ and $\{A_i\}_{k+1}^n$ are chains whose sum
under $\delta'$ is less than half that of $\{A_i\}_1^n$, the inductive
hypothesis applies to them and we may write
\begin{align*}
\delta'(A_1 \cup \cdots \cup A_{k-1})
  &\leq 2\bar{\delta}(A_1 \cup \cdots \cup A_{k-1})     &\text{inductive hypothesis}    \\
  &\leq 2\sum_{i=1}^{k-1}  \delta'(A_i)                 &\text{definition of }\bar{\delta}      \\
  &\leq \bar{\delta}(A)                                 &\text{by }\eqref{howiwishyouwerehere}  \\
  &< 2^{-N}
\end{align*}
Similarly $\delta'(A_{k+1} \cup \cdots \cup A_n) < 2^{-N}$, and
$\delta'(A_k) < 2^{-N}$ by \eqref{lucy}. So
\[ (A_1 \cup \cdots \cup A_{k-1}) \in C_{N+1} \text{ and } A_k \in C_{N+1} \text{ and } (A_{k+1}\cup\cdots\cup A_n) \in C_{N+1} \]
Our double-composition hypothesis gives
\[ (A_1 \cup \cdots \cup A_{k-1}) \cup A_k \cup (A_{k+1}\cup\cdots\cup A_n) \in C_N \]
And by monotonicity of $\delta'$,
\[ \delta'(A) \leq \delta'(A_1 \cup \cdots \cup A_n) \leq 2^{-N} \leq 2\bar{\delta}(A) \]
\end{enumerate}
\end{proof}

This characterizes the conformities generated by single pseudodiversities.
Later we will describe \emph{every} conformity in terms of the pseudodiversities
that generate them.

\subsection{Induced uniformities and completeness}

Given a conformity $\mathcal{C}$, we define its \defn{induced uniformity}
as the uniformity generated by the sets
\[ U_C = \{ (x,y) : \{ x,y \} \in C \} \]
for every $C\in\mathcal{C}$. It is straightforward to show that this is a
uniformity; since every singleton $\{x\}$ is in every $C\in\mathcal{C}$, we
have every pair $(x,x)$ in every generator of the induced uniformity, proving
(U1). Since $\{x,y\}=\{y,x\}$ we have (U2). Finally, (U3) follows from the
observation that whenever $\{x,y\}\in C\in\mathcal{C}$ and $\{y,z\}\in D\in\mathcal{C}$,
the set $D\circ D\in\mathcal{C}$ contains $\{x,y,z\}$. Then $\{x,z\}\in D\circ C$
by (C2). In other words, if $C\circ D \subseteq E$ in the conformity, then
$U_C\circ U_D\subseteq U_E$ in the induced uniformity. Thus (U3) is implied
by (C3).

\begin{thrm} Let $X$ be a set, $\{\delta_A\}_{A\in\mathcal{A}}$ a family of
diversities which generate a conformity $\mathcal{C}$. For each $\delta_A$,
write $d_A$ for its induced metric. Then the uniformity generated by the
metrics $\{d_A\}_{A\in\mathcal{A}}$ is exactly the induced uniformity of
$\mathcal{C}$.\label{sunking}\end{thrm}
\begin{proof} Denote by $\mathcal{U}_d$ the uniformity generated by
$\{d_A\}_{A\in\mathcal{A}}$, and by $\mathcal{U}_c$ the uniformity induced by
$\mathcal{C}$. A base for $\mathcal{C}$ is
\[ C_{\epsilon,A} = \{ F : \delta_A(F) < \epsilon \} \]
where $\epsilon$ ranges over $\mathbb{R}^+$ and $A$ ranges over $\mathcal{A}$.
Then a base for $\mathcal{C}$ is
\[ U_{C_{\epsilon,A}} = \{ (x,y) : \delta_A(\{x,y\}) < \epsilon \} = \{ (x,y) : d_A(x,y) < \epsilon \} \]
For the same $\epsilon,A$. But this is just the canonical base for $\mathcal{U}_d$!
\end{proof}

\begin{corl} Let $(X,\mathcal{C})$ be a conformity. Then $\mathcal{C}$ has
a countable base iff its induced uniformity does.\label{toobigtoobig}\end{corl}
\begin{proof} By Theorem \ref{idliketobe} $\mathcal{C}$ has a countable
base iff it is generated by a single pseudodiversity; by Theorem \ref{sunking}
this occurs iff the induced uniformity is generated by a single pseudometric.
A standard result \cite{bourbaki1989,kelley1955} gives that uniformities
with countable bases are exactly those generated by single pseudometrics.
\end{proof}


Next, we give some standard definitions. For a uniform space $(X,\mathcal{U})$, the
\defn{uniform topology} of $\mathcal{U}$ on $X$ is the smallest topology containing
the sets
\[ N(x,U) = \{ y : (x,y)\in U \} \]
for all $x\in X$, $U\in\mathcal{U}$. Notice that if $\mathcal{U}$ is generated by
a pseudometric, this coincides with the pseudometric topology.

With the same space $(X,\mathcal{U})$, we call a filter $\mathcal{F}$ on
$X$ \defn{Cauchy} if for every $U\in\mathcal{U}$, there is some $F\in\mathcal{F}$
with $F\times F\subseteq U$. We say that $\mathcal{F}$ \defn{converges} to some
$x\in X$ if every neighborhood of $x$ (in the uniform topology) is in $\mathcal{F}$. We
then call a uniformity \defn{complete} if every Cauchy filter converges. It can
be shown that a metric space is complete iff its generated uniformity is, and
that every uniformity can be embedded minimally (i.e., satisfying a universal
property with respect to uniformly continuous maps) in a complete uniformity
\cite{bourbaki1989,kelley1955}.

The analogous definitions for conformities are as follows.

Let $F$ be a filter on $X$. If for all $C\in\mathcal{C}$, there exists $f\in F$
with $\pfin{f}\subseteq C$, then $F$ is a \defn{Cauchy filter}. If $x\in X$ and for
all $C\in\mathcal{C}$ there exist $f\in F$ with $\pfin{f}\subseteq \{ A : A \cup
\{ x \} \in C \}$, then $F$ \defn{converges to} $x$. Finally, if every
Cauchy filter converges to some point in $X$, we say $\mathcal{C}$ is
\defn{complete}.

\begin{thrm} A pseudodiversity $(X, \delta)$ is complete iff its conformity
$\mathcal{C}$ is.\end{thrm}
\begin{proof} Suppose $(X, \delta)$ is complete and let $F$ be a Cauchy filter
on $X$. Then for every $\epsilon > 0$ there is some $f^\epsilon\in F$ so that
$\pfin{f^\epsilon}\subseteq \{A : \delta(A) < \epsilon\}$.
Take some sequence $\epsilon_n\to 0$, and define the sets $g^n\subseteq
X$ by $g^1=f^{\epsilon_1}$, $g^n=f^{\epsilon_n}\cap g^{\epsilon_{n-1}}$
for $n>1$.

Choose $x^n\in g^n$ for each $n$ to form a Cauchy sequence $\seq{x}$,
with some limit $x$. For any $\epsilon > 0$, find an integer $N$ so
that $\epsilon_n < \epsilon$ and $\delta(\{x_n, x\}) < \epsilon$ for
all $n \geq N$. Then if $a \in \pfin{f^{\epsilon_n}}$, so is $a\cup \{x_n\}$,
so that $\delta(a \cup \{x\}) \leq \delta(a \cup \{x_n\}) + \delta(\{x_n,x\})
< 2\epsilon$. We conclude that $F$ converges to $x$.

Conversely, suppose that every Cauchy filter converges in $\mathcal{C}$,
and let $\seq{x}$ be a Cauchy sequence in $(X,\delta)$. Choose the sets
$F_N = \seq{x}_N^\infty$. These sets generate a Cauchy filter with some
limit $x$. It is clear that $x_n \to x$.
\end{proof}

For any conformity $(X,\mathcal{C})$ generated by a diversity, the conformity
is complete iff its diversity is iff the diversity's uniformity is
\cite{bourbaki1989,kelley1955}; thus completeness of the conformity is equivalent
to completeness of its induced uniformity. In fact, this is true in general,
as the next theorem shows.

\begin{thrm} Let $(X,\mathcal{C})$ be a conformity with complete induced uniformity
$\mathcal{U}$. Then $\mathcal{C}$ is complete.\label{sdsdsd}\end{thrm}
\begin{proof} Suppose that $\mathcal{U}$ is complete, and let $\mathcal{F}$ be
a Cauchy filter with respect to $\mathcal{C}$. Then $\mathcal{F}$ is also Cauchy
with respect to $\mathcal{U}$, since for all $C\in \mathcal{C}$, we have
$\{\{x,y\} :x,y\in F \} \subseteq \pfin{F} \subseteq C$ for some $F\in\mathcal{F}$;
then $F\times F\subseteq U_C$. Thus $\mathcal{F}$ converges in $\mathcal{U}$ to
some element $x$, and we claim that it also converges to $x$ in $\mathcal{C}$.
To this end, fix $C\in \mathcal{C}$. Choose $D\in\mathcal{C}$ so that $D\circ D
\subseteq C$. and $F\in\mathcal{F}$ so that (a) $y \in F$ whenever $(x,y)\in U_D$
and (b) $\pfin{F}\subseteq D$. Then for all $A\in \pfin{F}$, $A\cup \{x\}\in C$.
(If $A=\varnothing$, $A\cup\{x\}\in C$ trivially. Otherwise pick $y\in A$, and
we will have $A\in D$ and $\{x,y\}\in D$, so that $A\cup \{x,y\} = A\cup\{x\}
\in C$.)\end{proof}

We end this section with two open questions:
\begin{enumerate}
\item Does the converse to Theorem \ref{sdsdsd} hold; that is, if a conformity
$(X,\mathcal{C})$ is complete, must its induced uniformity be?

\item We saw in Section \ref{sect-comp} that for any diversity $(X,\delta)$ it is
possible to embed $X$ in a complete diversity which was universal, meaning
that any uniformly continuous map from $X$ to a complete diversity factored
through the embedding. It is shown in \cite{kelley1955} that every uniformity
can be embedded in a complete uniformity. This embedding is also universal.

Is there a notion of universal completion for conformities?
\end{enumerate}

\subsection{Diversities of conformities}

Not every conformity has a countable base. For example, let $X$ be the space
of functions $f:[0,1]\to[0,1]$, and consider the ``pointwise convergence''
conformity generated by the sets
\[ C^x_\epsilon = \{ \{ f_1,\ldots,f_n \} : \diam(\{ f_1(x), \ldots, f_n(x) \}) < \epsilon \} \]
for every $\epsilon > 0$, $x\in [0,1]$. This conformity has no countable base
by Corollary \ref{toobigtoobig}, since its induced uniformity does not have
a countable base \cite{MO134421}. Thus by Theorem \ref{idliketobe}
it is not generated by any pseudodiversity.

In this section we will show that every conformity is generated by the
collection of
pseudodiversities which are uniformly continuous with respect to it, in an
appropriate sense. In the case of uniformities this is done by constructing
a so-called product uniformity; given a uniformity on a set $X$, the product
uniformity is constructed on $X\times X$. Then a given pseudometric $d$ may
or may not be uniformly continuous from the product uniformity to the
Euclidean uniformity on $\mathbb{R}$. It can be proven \cite{bourbaki1989,kelley1955}
that a uniformity $\mathcal{U}$ is exactly the uniformity generated by all
pseudometrics which are uniformly continuous from its product uniformity.

Since pseudodiversities are functions on finite sets rather than pairs,
given a conformity on a set $X$ we seek a conformity on $\pfin{X}$ from
which to judge uniform continuity of pseudodiversities.

In fact, such a conformity exists for which we can prove the same result:
given a conformity $(X, \mathcal{C})$, define the \defn{power conformity}
$\mathcal{C}^P$ as the conformity on $\pfin{X}$ generated by the sets
\begin{equation}
 \label{acidisgroovy}
 C_u = \left\{ \{ A_1, \ldots, A_n \} : n \leq 1\text{ or } \bigcup_{i=1}^n A_i \in u \right\}
\end{equation}
where $u$ ranges over all members of $\mathcal{C}$.

\begin{lmma} A power conformity is a conformity.\end{lmma}
\begin{proof} First, the $C_u$'s form a filter base since $C_u\cap C_v
= C_{u\cap v} \in \mathcal{C}^P$ for any $C_u$, $C_v\in\mathcal{C}^P$.
For all $A\in\pfin{X}$, $\{A\}$ is in every $C_u$ by definition. It is
immediate that whenever $\{A_i\}$ is in $C_u$, so is every subset of
$\{A_i\}$.

Finally, every $C_u$ has a $C_v$ with $C_v\circ C_v \subseteq C_u$: choose
$v$ with $v\circ v \subseteq u$ in $\mathcal{C}$. If $\{A_i\}_{i=1}^n$,
$\{B_i\}_{i=1}^m$ are in $C_v$ with some $A_i$ equal to some $B_j$, then
(a) $m \leq 1$ and $n \leq 1$, so their union has at most one element and
therefore must lie in $C_u$, (b) exactly one of $m \leq 1$ or $n \leq 1$,
in which case one of the sets is a subset of the other, so their union lies
in $v$ (and therefore $u$), or (c) $m > 1$ and $n > 1$, so the sets
$\bigcup_{i=1}^n A_i$ and $\bigcup_{i=1}^m B_i$ are sets in $v$ with nonempty
intersection. Then since $v\circ v \subseteq u$, their union lies in $u$. In
every case we have $\{A_i\}_{i=1}^n\cup\{B_i\}_{i=1}^m \in C_u$.
\end{proof}

\begin{thrm} Let $(X, \mathcal{C})$ be a conformity. A pseudodiversity
$\delta$ is uniformly continuous from the $\mathcal{C}^P$ to $(\mathbb{R},
\diam)$ iff the set $V_\epsilon = \{ A: \delta(A) < \epsilon \}$ is in
$\mathcal{C}$ for each $\epsilon > 0$. \label{underthesea}
\end{thrm}
\begin{proof} First, suppose that every $V_\epsilon$ is in $\mathcal{C}$.
For each $\epsilon > 0$, the set
\[ C_u = \left\{ \{ A_1, \ldots, A_n \} : n \leq 1\text{ or } \delta\left(\bigcup_{i=1}^n A_i\right) < \epsilon \right\} \]
is in $\mathcal{C}^P$. (Notice it has the form of \eqref{acidisgroovy}
with $u = V_\epsilon$.) Let $\{A, B\} \in C_\epsilon$; then $\delta(A)
\leq \delta(A \cup B) < \epsilon$ and similarly $\delta(B) < \epsilon$.
Thus $|\delta(A) - \delta(B)| < \epsilon$, so $\delta$ is uniformly
continuous.

Conversely, suppose $\delta$ is uniformly continuous. Then for any
$\epsilon > 0$, there exists some $u \in \mathcal{C}$, such that every
$\{A_1,\ldots,A_n\} \in C_u$ satisfies $\sup_{i,j} |\delta(A_i) - \delta(A_j)| < \epsilon$.
Since for any $A\in u$, the set $\{A,\varnothing\}$ lies in $C_u$,
this implies that $\delta(A) < \epsilon$, which in turn implies that
$u\subseteq V_\epsilon$, which finally implies that $V_\epsilon$ is
in $\mathcal{C}$.
\end{proof}

\begin{corl} Every conformity is generated by the pseudodiversities
which are uniformly continuous from its power conformity to $(\mathbb{R},
\diam)$.\end{corl}
\begin{proof} Let $\mathcal{C}$ be a conformity, $\mathcal{D}$ the conformity
generated by the pseudodiversities which are uniformly continuous from the
power conformity to $(\mathbb{R},\diam)$. By Theorem \ref{idliketobe} we have
$\mathcal{C}\subseteq\mathcal{D}$, since every member $u$ of $\mathcal{C}$ is
in a countably-based subconformity of $\mathcal{C}$. (Take $u_0 = u$, $u_i$
such that $u_i\circ u_i\subseteq u_{i-1}$, $i>0$ as a base.)

Then by Theorem \ref{underthesea}, every pseudodiversity which is uniformly
continuous generates a subset of $\mathcal{C}$; that is, $\mathcal{D}\subseteq
\mathcal{C}$.
\end{proof}

We saw at the beginning of this section that some conformities can be generated
by sets of the form
$\left\{\delta^{-1}_\alpha[0, \epsilon]\right\}_{\alpha\in\mathcal{A},\epsilon>0}$,
where $\mathcal{A}$ is some collection of pseudodiversities, $\epsilon > 0$.
What we have just shown is that all conformities are generated in this
way, so that we may \emph{define} a conformity as a filter generated in
this way by some collection of diversities.

\section{Category theory}

In \cite{bryant+tupper2012}, Bryant and Tupper introduced the category \textbf{Dvy}
whose objects are diversities and morphisms nonexpansive maps (functions $f$ between
diversities $(X,\delta)$ and $(Y,\rho)$ such that $\rho(f(A))\leq \delta(A)$ for
all finite $A\subseteq X$). Compare with \textbf{Met} \cite{adamek+herrlich+strecker1990},
whose objects are metric spaces and morphisms nonexpansive maps (functions $f$
between metric spaces $(X,d)$ and $(Y,p)$ such that $p(f(x), f(y)) \leq d(x,y)$
for all $x,y\in X$).

It is not hard to see that for both metric spaces and diversities, nonexpansive
maps are uniformly continuous. In the metric case, they are also continuous.

We introduce the category \textbf{Conf}, whose objects are conformities and
morphisms uniformly continuous functions. Compare with \textbf{Unif} \cite{adamek+herrlich+strecker1990},
whose objects are uniformities and morphisms uniformly continuous functions.

We also recall \textbf{Top}, whose objects are topological spaces and morphisms
continous maps, and \textbf{CAT}, whose objects are categories and morphisms
are functors (maps between categories which preserve composition).

With these categories in hand, we can summarize the relationships between
diversities, conformities and metric spaces by observing that the maps in
the following diagram in \textbf{CAT} are functors, and that the diagram
as a whole commutes.

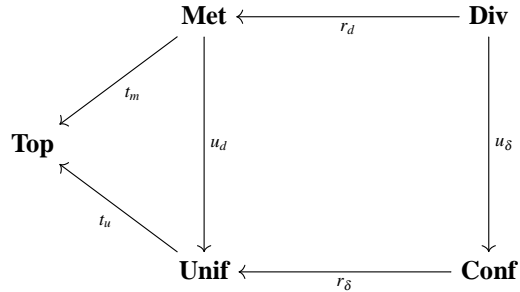
\begin{center}
\begin{tikzcd}[column sep=large,row sep=large]
	~	& \textbf{Met} \arrow{dl}{t_m} \arrow{dd}{u_d}	&& \textbf{Div} \arrow{ll}{r_d} \arrow{dd}{u_\delta}	\\
  \textbf{Top}	&						&& \\
	~	& \textbf{Unif} \arrow{ul}{t_u}			&& \textbf{Conf} \arrow{ll}{r_\delta}
\end{tikzcd}
\end{center}
where:
\begin{itemize}
\item $r_\delta$ maps conformities to their induced uniform spaces;
\item $r_d$ maps diversities to their induced metric spaces;
\item $u_\delta$ maps diversities to the conformities that they generate;
\item $u_d$ maps metric spaces to the uniform spaces that they generate;
\item $t_m$ maps metric spaces to their metric topologies;
\item and $u_m$ maps uniform spaces to their uniform topologies.
\end{itemize}
Notice that each functor leaves the underlying sets unchanged, e.g. $u_d$ maps
a metric space $(X,d)$ to a uniform space $(X, \mathcal{U})$. The morphisms
are also unchanged as functions, e.g., a nonexpansive map $f:X\to Y$ in \textbf{Met}
is considered a continuous map in \textbf{Top} under $t_m$ and a uniformly
continous map in \textbf{Unif} under $u_d$, but it is the same function from
the set $X$ to the set $Y$ in all cases.

%
%
%
\bibliographystyle{amsplain}
\bibliography{asp}

\end{document}